\documentclass[12pt, reqno]{amsart}
\usepackage{amssymb,latexsym,amsmath,amsfonts}
\usepackage{latexsym}
\usepackage[mathscr]{eucal}
\usepackage{bm}

\def ~{\hspace{1mm}}

\def\textmatrix#1&#2\\#3&#4\\{\bigl({#1 \atop #3}\ {#2 \atop #4}\bigr)}
\def\dispmatrix#1&#2\\#3&#4\\{\left({#1 \atop #3}\ {#2 \atop #4}\right)}
\newcommand{\beg}{\begin{equation}}
\newcommand{\eeg}{\end{equation}}
\newcommand{\ben}{\begin{eqnarray*}}
\newcommand{\een}{\end{eqnarray*}}

\newtheorem{thm}{Theorem}[section]

\newtheorem{lem}[thm]{Lemma}

\newtheorem{prop}[thm]{Proposition}
\numberwithin{equation}{section} \theoremstyle{definition}
\newtheorem{defn}[thm]{Definition}

\def\textmatrix#1&#2\\#3&#4\\{\bigl({#1 \atop #3}\ {#2 \atop #4}\bigr)}
\def\dispmatrix#1&#2\\#3&#4\\{\left({#1 \atop #3}\ {#2 \atop #4}\right)}

\begin{document}
\title[The failure of rational dilation on the tetrablock]
{The failure of rational dilation on the tetrablock}

\author[Sourav Pal]{Sourav Pal}
\address[Sourav Pal]{Department of Mathematics,
Indian Institute of Technology Bombay, Powai, Mumbai - 400076,
India.} \email{sourav@math.iitb.ac.in , souravmaths@gmail.com}



\thanks{ This work was supported in part by the Center
for Advanced Studies in Mathematics at Ben-Gurion University of
the Negev, Israel when the author was visiting the Mathematics
Department of BGU during 2012 - 2014. At present the author's work
is supported by INSPIRE Faculty Award of DST, India.}

\begin{abstract}
We show by a counter example that rational dilation fails on the
tetrablock, a polynomially convex and non-convex domain in
$\mathbb C^3$ defined as
\[
\mathbb E = \{ (x_1,x_2,x_3)\in\mathbb C^3\,:\,
1-zx_1-wx_2+zwx_3\neq 0 \textup{ whenever } |z|\leq 1, |w|\leq 1
\}.
\]
A commuting triple of operators $(T_1,T_2,T_3)$ for which
the closed tetrablock $\overline{\mathbb E}$ is a spectral set, is
called an $\mathbb E$-contraction. For an $\mathbb E$-contraction
$(T_1,T_2,T_3)$, the two operator equations
\[
T_1-T_2^*T_3=D_{T_3}X_1D_{T_3} \textup{ and } T_2-T_1^*T_3=
D_{T_3}X_2D_{T_3}, \quad D_{T_3}=(I-T_3^*T_3)^{\frac{1}{2}},
\]
have unique solutions $A_1,A_2$ on $\mathcal
D_{T_3}=\overline{Ran} D_{T_3}$ and they are called the
fundamental operators of $(T_1,T_2,T_3)$. For a particular class
of $\mathbb E$-contractions, we prove it necessary for the
existence of rational dilation that the corresponding fundamental
operators $A_1,A_2$ satisfy
\begin{equation}\label{abstract}
A_1A_2=A_2A_1 \textup{ and } A_1^*A_1-A_1A_1^*=A_2^*A_2-A_2A_2^*.
\end{equation}
Then we construct an $\mathbb E$-contraction from that particular
class which fails to satisfy (\ref{abstract}). We produce a
concrete functional model for pure $\mathbb E$-isometries, a class
of $\mathbb E$-contractions analogous to the pure isometries in
one variable. The fundamental operators play the main role in this
model.
\end{abstract}

\maketitle

\section{Introduction}

\noindent Let $X$ be a compact subset of $\mathbb C^n$ and let
$\mathcal R(X)$ denote the algebra of all rational functions on
$X$, that is, all quotients $p/q$ of polynomials $p,q$ for which
$q$ has no zeros in $X$. The norm of an element $f$ in $\mathcal
R(X)$ is defined as
\[
\|f\|_{\infty, X}=\sup \{|f(\xi)|\;:\; \xi \in X  \}.
\]
Also for each $k\geq 1$, let $\mathcal R_k(X)$ denote the algebra
of all $k \times k$ matrices over $\mathcal R(X)$. Obviously each
element in $\mathcal R_k(X)$ is a $k\times k$ matrix of rational
functions $F=(f_{i,j})$ and we can define a norm on $\mathcal
R_k(X)$ in the canonical way
\[
\|F\|=\sup \{ \|F(\xi)\|\;:\; \xi\in X \},
\]
thereby making $\mathcal R_k(X)$ into a non-commutative normed
algebra. Let $\underline{T}=(T_1,\cdots,T_n)$ be an $n$-tuple of
commuting operators on a Hilbert space $\mathcal H$. The set $X$
is said to be a \textit{spectral set} for $\underline T$ if the
Taylor joint spectrum $\sigma (\underline T)$ of $\underline T$ is
a subset of $X$ and
\begin{equation}\label{defn1}
\|f(\underline T)\|\leq \|f\|_{\infty, X}\,, \textup{ for every }
f\in \mathcal R(X).
\end{equation}
Here $f(\underline T)$ can be interpreted as $p(\underline
T)q(\underline T)^{-1}$ when $f=p/q$. Moreover, $X$ is said to be
a \textit{complete spectral set} if $\|F(\underline T)\|\leq
\|F\|$ for every $F$ in $\mathcal R_k(X)$, $k=1,2,\cdots$.\\

Let $\mathcal A(X)$ be the algebra of continuous complex-valued
functions on $X$ which separates the points of $X$. A
\textit{boundary} for $\mathcal A(X)$ is a closed subset $F$ of
$X$ such that every function in $\mathcal A(X)$ attains its
maximum modulus on $F$. It follows from the theory of uniform
algebras that if $bX$ is the intersection of all the boundaries of
$X$ then $bX$ is a boundary for $\mathcal A(X)$ (see Theorem 9.1
of \cite{wermer}). This smallest boundary $bX$ is called the
$\check{\textup{S}}$\textit{ilov boundary} relative to the
algebra $\mathcal A(X)$ .\\

A commuting $n$-tuple of operators $\underline T$ that has $X$ as
a spectral set, is said to have a \textit{rational dilation} or
\textit{normal} $bX$-\textit{dilation} if there exists a Hilbert
space $\mathcal K$, an isometry $V:\mathcal H \rightarrow \mathcal
K$ and an $n$-tuple of commuting normal operators $\underline
N=(N_1,\cdots,N_n)$ on $\mathcal K$ with $\sigma(\underline
N)\subseteq bX$ such that
\begin{equation}\label{rational-dilation}
f(\underline T)=V^*f(\underline N)V, \textup{ for every } f\in
\mathcal R(X).
\end{equation}

One of the important discoveries in operator theory is Sz.-Nagy's
unitary dilation for a contraction, \cite{nagy1}, which opened a
new horizon by announcing the success of rational dilation on the
closed unit disc of $\mathbb C$. Since then one of the main aims
of operator theory has been to determine the success or failure of
rational dilation on the closure of a bounded domain in $\mathbb
C^n$. It is evident from the definitions that if $X$ is a complete
spectral set for $\underline T$ then $X$ is a spectral set for
$\underline T$. A celebrated theorem of Arveson states that
$\underline T$ has a normal $bX$-dilation if and only if $X$ is a
complete spectral set for $\underline T$ (Theorem 1.2.2 and its
corollary, \cite{sub2}). Therefore, the success or failure of
rational dilation is equivalent to asking whether the fact that
$X$ is a spectral set for $\underline T$ automatically turns $X$
into a complete spectral set for $\underline T$. History witnessed
an affirmative answer to this question given by Agler when $X$ is
an annulus \cite{agler-ann} and by Ando when $X=\overline{\mathbb
D^2}$ \cite{ando}. Agler, Harland and Raphael have produced, by
machine computation, an example of a triply connected domain in
$\mathbb C$ where the answer is negative \cite{ahr}. Dritschel and
M$^{\textup{c}}$Cullough also gave a negative answer to that
question when $X$ is an arbitrary triply connected domain
\cite{DM}. Parrott showed by a counter example \cite{parrott} that
rational dilation fails on the closed tridisc $\overline{\mathbb
D^3}$. Also recently we have success of rational dilation on the
closed symmetrized bidisc $\Gamma$ \cite{ay-jfa, tirtha-sourav,
sourav}, where $\Gamma$ is defined as
\begin{equation}\label{gamma}
\Gamma =\{ (z_1+z_2,z_1z_2)\,:\, |z_1|\leq 1, |z_2|\leq 1 \}.
\end{equation}

In this article, we show that rational dilation fails when $X$ is
the closure of the tetrablock $\mathbb E$, a polynomially convex,
non-convex and inhomogeneous domain in $\mathbb C^3$, defined as
\[
\mathbb E=\{ (x_1,x_2,x_3)\in\mathbb C^3\,:\,
1-zx_1-wx_2+zwx_3\neq 0 \textup{ whenever } |z|\leq 1, |w|\leq 1
\}.
\]
This domain has attracted the attention of a number of
mathematicians
\cite{awy,awy-cor,young,EZ,EKZ,Zwo,tirtha,tirtha-sau,sourav2}
because of its relevance to $\mu$-synthesis and $H^{\infty}$
control theory. The following result from \cite{awy} (Theorem 2.4,
part-(9)) characterizes points in $\mathbb E$ and
$\overline{\mathbb E}$ and provides a geometric description of the
tetrablock.
\begin{thm}\label{thm1}
A point $(x_1,x_2,x_3)\in \mathbb C^3$ is in $\overline{\mathbb
E}$ if and only if $|x_3|\leq 1$ and there exist $\beta_1,\beta_2
\in \mathbb C$ such that $|\beta_1|+|\beta_2| \leq 1$ and
$x_1=\beta_1 + \bar{\beta_2}x_3, \quad
x_2=\beta_2+\bar{\beta_1}x_3$.
\end{thm}
It is evident from the above result that the tetrablock lives
inside the tridisc $\mathbb D^3$. The distinguished boundary
(which is same as the $\check{\textup{S}}$\textit{ilov} boundary)
of the tetrablock was determined in \cite{awy} (see Theorem 7.1 of
\cite{awy}) to be the set
\begin{align*}
b \overline{\mathbb E} &=\{(x_1,x_2,x_3)\in \mathbb C^3\,:\,
x_1=\bar{x_2}x_3,|x_2|\leq 1, |x_3|=1 \}\\& =\{ (x_1,x_2,x_3)\in
\overline{\mathbb E}\,:\, |x_3|=1 \}.
\end{align*}

In \cite{tirtha}, Bhattacharyya introduced the study of commuting
operator triples that have $\overline{\mathbb E}$ as a spectral
set. There such a triple was called a \textit{tetrablock
contraction}. As a notation is always convenient, we shall call
such a triple an $\mathbb E$-contraction. So we are led to the
following definition:
\begin{defn}
A triple of commuting operators $(T_1,T_2,T_3)$ on a Hilbert space
$\mathcal H$ for which $\overline{\mathbb E}$ is a spectral set is
called an $\mathbb E$-$contraction$.
\end{defn}
Since the tetrablock lives inside the tridisc, an $\mathbb
E$-contraction consists of commuting contractions. Evidently
$(T_1^*,T_2^*,T_3^*)$ is an $\mathbb E$-contraction when
$(T_1,T_2,T_3)$ is an $\mathbb E$-contraction. We briefly recall
from the literature the special classes of $\mathbb
E$-contractions which are analogous to uniteries, isometries and
co-isometries in one variable operator theory.
\begin{defn}
Let $T_1,T_2,T_3$ be commuting operators on a Hilbert space
$\mathcal H$. We say that $(T_1,T_2,T_3)$ is
\begin{itemize}
\item [(i)] an $\mathbb E$-\textit{unitary} if $T_1,T_2,T_3$ are
normal operators and the joint spectrum $\sigma_T(T_1,T_2,T_3)$ is
contained in $b\overline{\mathbb E}$ ; \item [(ii)] an $\mathbb
E$-\textit{isometry} if there exists a Hilbert space $\mathcal K$
containing $\mathcal H$ and an $\mathbb E$-unitary
$(\tilde{T_1},\tilde{T_2},\tilde{T_3})$ on $\mathcal K$ such that
$\mathcal H$ is a common invariant subspace of $T_1,T_2,T_3$ and
that $T_i=\tilde{T_i}|_{\mathcal H}$ for $i=1,2,3$; \item [(iii)]
an $\mathbb E$-\textit{co-isometry} if $(T_1^*,T_2^*,T_3^*)$ is an
$\mathbb E$-isometry.
\end{itemize}
\end{defn}
 Moreover, an
$\mathbb E$-isometry $(T_1,T_2,T_3)$ is said to be \textit{pure}
if $T_3$ is a pure isometry, i.e, if ${T_3^*}^n \rightarrow 0$
strongly as $n \rightarrow \infty$. We accumulate some results
from the
literature in section 2 and they will be used in sequel.\\

It is clear that a rational dilation of an $\mathbb E$-contraction
$(T_1,T_2,T_3)$ is nothing but an $\mathbb E$-unitary dilation of
$(T_1,T_2,T_3)$, that is, an $\mathbb E$-unitary $N=(N_1,N_2,N_3)$
that dilates $T$ by satisfying (\ref{rational-dilation}).
Similarly an $\mathbb E$-isometric dilation of $T=(T_1,T_2,T_3)$
is an $\mathbb E$-isometry $V=(V_1,V_2,V_3)$ that satisfies
(\ref{rational-dilation}). In Theorem 3.5 in \cite{tirtha}, it was
shown that for every $\mathbb E$-contraction $(T_1,T_2,T_3)$ there
were two unique operators $A_1,A_2$ in $\mathcal L(\mathcal
D_{T_3})$ such that
\[
T_1-T_2^*T_3=D_{T_3}A_1D_{T_3}\, , \;
T_2-T_1^*T_3=D_{T_3}A_2D_{T_3}\,.
\]
Here $D_{T_3}=(I-T_3^*T_3)^{\frac{1}{2}}$ and $\mathcal
D_{T_3}=\overline{Ran}\,D_{T_3}$ and $\mathcal L(\mathcal H)$, for
a Hilbert space $\mathcal H$, always denotes the algebra of
bounded operators on $\mathcal H$. An explicit $\mathbb
E$-isometric dilation was constructed for a particular class of
$\mathbb E$-contractions in \cite{tirtha} (see Theorem 6.1 in
\cite{tirtha}) and $A_1,A_2$ played the fundamental role in that
explicit construction of dilation.  For their pivotal role in the
dilation, $A_1$ and $A_2$ were called the
\textit{fundamental operators} of $(T_1,T_2,T_3)$.\\

In section 4, we produce a set of necessary conditions for the
existence of rational dilation for a class of $\mathbb
E$-contractions. Indeed, in Proposition \ref{ultimate}, we show
that if $(T_1,T_2,T_3)$ is an $\mathbb E$-contraction on $\mathcal
H_1 \oplus \mathcal H_1$ for some Hilbert space $\mathcal H_1$,
satisfying
\begin{itemize}
 \item[(i)] $ Ker (D_{T_3})=\mathcal H_1\oplus \{0\} \textup{ and } \mathcal D_{T_3} =
 \{0\}\oplus \mathcal H_1$
\item[(ii)] $T_3(\mathcal D_{T_3})=\{0\}$ \textup{and} $T_3
Ker(D_{T_3})\subseteq \mathcal D_{T_3}$
\end{itemize}
 and if $A_1,A_2$ are the fundamental operators of $(T_1,T_2,T_3)$, then for the existence
 of an $\mathbb E$-isometric dilation of
$(T_1^*,T_2^*,T_3^*)$ it is necessary that
\begin{equation}\label{nece-eqn}
[A_1,A_2]=0 \textup{ and }[A_1^*,A_1]=[A_2^*,A_2].
\end{equation}
Here $[S_1,S_2]=S_1S_2-S_2S_1$, for any two operators $S_1,S_2$.
In section 5, we construct an example of an $\mathbb
E$-contraction that satisfies the hypotheses of Proposition
\ref{ultimate} but fails to satisfy (\ref{nece-eqn}). This
concludes
the failure of rational dilation on the tetrablock.\\

The proof of Proposition \ref{ultimate} depends heavily upon a
functional model for pure $\mathbb E$-isometries which we provide
in Theorem \ref{model1}. There is an Wold type decomposition for
an $\mathbb E$-isometry (see Theorem \ref{ti}) that splits an
$\mathbb E$-isometry into two parts of which one is an $\mathbb
E$-unitary and the other is a pure $\mathbb E$-isometry. Theorem
\ref{tu} describes the structure of an $\mathbb E$-unitary.
Therefore, a concrete model for pure $\mathbb E$-isometries gives
a complete description of an $\mathbb E$-isometry. In Theorem
\ref{model1}, we show that a pure $\mathbb E$-isometry
$(\hat{T_1},\hat{T_2},\hat{T_3})$ can be modelled as a commuting
triple of Toeplitz operators $(T_{A_1^*+A_2z},T_{A_2^*+A_1z},T_z)$
on the vectorial Hardy space $H^2(\mathcal D_{\hat{T_3^*}})$,
where $A_1$ and $A_2$ are the fundamental operators of the
$\mathbb E$-co-isometry $(\hat{T_1}^*,\hat{T_2}^*,\hat{T_3}^*)$.
The converse is also true, that is, every such triple of commuting
contractions $(T_{A+Bz},T_{B^*+A^*z},T_z)$ on a vectorial Hardy
space is a pure $\mathbb E$-isometry.

\section{Preliminary results}

We begin with a lemma that simplifies the definition of $\mathbb
E$-contraction.

\begin{lem} \label{simpler}
A commuting triple of bounded operators $(T_1,T_2,T_3)$ is an
$\mathbb E$-contraction if and only if $\| f (T_1,T_2,T_3) \| \leq
\| f \|_{\infty, \overline{\mathbb E}}$ for any holomorphic
polynomial $f$ in three variables.
\end{lem}
This actually follows from the fact that $\overline{\mathbb E}$ is
polynomially convex. For a proof to this lemma see Lemma 3.3 of
\cite{tirtha}. The following theorem gives a set of
characterization for $\mathbb E$-unitaries (Theorem 5.4 of
\cite{tirtha}).

\begin{thm}\label{tu}
Let $\underline N = (N_1, N_2, N_3)$ be a commuting triple of
bounded operators. Then the following are equivalent.

\begin{enumerate}

\item $\underline N$ is an $\mathbb E$-unitary,

\item $N_3$ is a unitary and $\underline N$ is an $\mathbb
E$-contraction,

\item $N_3$ is a unitary, $N_2$ is a contraction and $N_1 = N_2^*
N_3$.
\end{enumerate}
\end{thm}

\noindent Here is a structure theorem for the $\mathbb
E$-isometries.
\begin{thm} \label{ti}

Let $\underline V = (V_1, V_2, V_3)$ be a commuting triple of
bounded operators. Then the following are equivalent.

\begin{enumerate}

\item $\underline V$ is an $\mathbb E$-isometry.

\item $V_3$ is an isometry and $\underline V$ is an $\mathbb
E$-contraction.

\item $V_3$ is an isometry, $V_2$ is a contraction and $V_1=V_2^*
V_3$.

\item (\textit{Wold decomposition}) $\mathcal H$ has a
decomposition $\mathcal H=\mathcal H_1\oplus \mathcal H_2$ into
reducing subspaces of $V_1,V_2,V_3$ such that $(V_1|_{\mathcal
H_1},V_2|_{\mathcal H_1},V_3|_{\mathcal H_1})$ is an $\mathbb
E$-unitary and $(V_1|_{\mathcal H_2},V_2|_{\mathcal
H_2},V_3|_{\mathcal H_2})$ is a pure $\mathbb E$-isometry.

\end{enumerate}

\end{thm}
See Theorem 5.6 and Theorem 5.7 of \cite{tirtha} for a proof.

\section{A functional model for pure $E$-isometries}

Let us recall that the {\em numerical radius} of an operator $T$
on a Hilbert space $\mathcal{H}$ is defined by
\[
\omega(T) = \sup \{|\langle Tx,x \rangle|\; : \;
\|x\|_{\mathcal{H}}= 1\}.
\]
It is well known that
\begin{eqnarray}\label{nradius}
r(T)\leq \omega(T)\leq \|T\| \textup{ and } \frac{1}{2}\|T\|\leq
\omega(T)\leq \|T\|, \end{eqnarray} where $r(T)$ is the spectral
radius of $T$. We state a basic lemma on numerical radius whose
proof is a routine exercise. We shall use this lemma in sequel.
\begin{lem} \label{basicnrlemma}
The numerical radius of an operator $T$ is not greater than one if
and only if  Re $\beta T \leq I$ for all complex numbers $\beta$
of modulus $1$.
\end{lem}

We recall from section 1, the existence-uniqueness theorem
(\cite{tirtha}, Theorem 3.5) for the fundamental operators of an
$\mathbb E$-contraction.

\begin{thm}\label{funda-exist-unique}
Let $(T_1,T_2,T_2)$ be an $\mathbb E$-contraction. Then there are
two unique operators $A_1,A_2$ in $\mathcal L(\mathcal D_{T_3})$
such that
\begin{equation}\label{basiceqn}
T_1-T_2^*T_3=D_{T_3}A_1D_{T_3}\textup{ and }
T_2-T_1^*T_3=D_{T_3}A_2D_{T_3}. \end{equation} Moreover,
$\omega(A_1+zA_2)\leq 1$ for all $z\in\overline{\mathbb D}$.
\end{thm}
As we mentioned in Section 1 that these two unique operators
$A_1,A_2$ are called the fundamental operators of $(T_1,T_2,T_3)$.
The following theorem gives a concrete model for pure $\mathbb
E$-isometries in terms of Toeplitz operators on a vectorial Hardy
space.

\begin{thm}\label{model1}
Let $(\hat{T_1},\hat{T_2},\hat{T_3})$ be a pure $\mathbb
E$-isometry acting on a Hilbert space $\mathcal H$ and let
$A_1,A_2$ denote the corresponding fundamental operators. Then
there exists a unitary $U:\mathcal H \rightarrow H^2(\mathcal
D_{{\hat{T_3}}^*})$ such that
\[
\hat{T_1}=U^*T_{\varphi}U,\quad \hat{T_2}=U^*T_{\psi}U \textup{
and } \hat{T_3}=U^*T_zU,
\]
where $\varphi(z)= G_1^*+G_2z,\,\psi(z)= G_2^*+G_1z, \quad
z\in\mathbb D$ and $G_1=UA_1U^*$ and $G_2=UA_2U^*$. Moreover,
$A_1,A_2$ satisfy
\begin{enumerate}
\item $[A_1,A_2]=0\,;$ \item $[A_1^*,A_1]=[A_2^*,A_2] \,;$ and
\item $\|A_1^*+A_2z\|\leq 1$ for all $z\in {\mathbb D}$.
\end{enumerate}
Conversely, if $A_1$ and $A_2$ are two bounded operators on a
Hilbert space $E$ satisfying the above three conditions, then
$(T_{A_1^*+A_2z},T_{A_2^*+A_1z},T_z)$ on $H^2(E)$ is a pure
$\mathbb E$-isometry.
\end{thm}

\begin{proof}
Suppose that $(\hat{T_1},\hat{T_2},\hat{T_3})$ is a pure $\mathbb
E$-isometry. Then $\hat{T_3}$ is a pure isometry and it can be
identified with the Toeplitz operator $T_z$ on $H^2(\mathcal
D_{{\hat{T_3}}^*})$. Therefore, there is a unitary $U$ from
$\mathcal H$ onto $H^2(\mathcal D_{{\hat{T_3}}^*})$ such that
$\hat{T_3}=U^*T_zU$. Since for $i=1,2,\, \;\hat{T_i}$ is a
commutant of $\hat{T_3}$, there are two multipliers $\varphi,\,
\psi$ in $H^{\infty}(\mathcal L(\mathcal D_{{\hat{T_3}}^*}))$ such
that
$\hat{T_1}=U^*T_{\varphi}U$ and $\hat{T_2}=U^*T_{\psi}U$.\\

\noindent \textit{Claim.} If $(V_1,V_2,V_3)$ on a Hilbert space
$\mathcal H_1$ is an $\mathbb E$-isometry then $V_2=V_1^*V_3$.

\noindent \textit{Proof of Claim.} Let $(V_1,V_2,V_3)$ be the
restriction of an $\mathbb E$-unitary $(N_1,N_2,N_3)$ to the
common invariant subspace $\mathcal H_1$. By part-(3) of Theorem
\ref{tu} we have that $N_3$ is a unitary and $N_1=N_2^*N_3$.
Therefore, $N_1^*=N_3^*N_2$ and hence $N_1^*=N_2N_3^*$ by an
application of Fuglede's theorem, \cite{Fuglede}, which states
that if a normal operator $N$ commutes with a bounded operator $T$
then it commutes with $T^*$ too. Also since $N_3$ is a unitary we
have that $N_2=N_1^*N_3$. Now $\mathcal H_1$ is an invariant
subspace for $N_2$ and thus $\mathcal H_1$ is invariant under
$N_1^*N_3$. So $V_2=N_2|_{\mathcal H_1}=N_1^*N_3|_{\mathcal H_1}$.
Again $\mathcal H_1$ is invariant under $N_3$. Therefore,
$N_1^*(N_3(\mathcal H_1))\subseteq \mathcal H_1$. So we have that
$P_{\mathcal H_1}N_1^*|_{N_3(\mathcal H_1)}=N_1^*|_{N_3(\mathcal
H_1)}$. Again $V_1^*=P_{\mathcal H_1}N_1^*|_{\mathcal H_1}$.
Therefore, $N_1^*N_3|_{\mathcal H_1}=V_1^*V_3$. So, we have that
$V_2=V_1^*V_3$.\\

We apply this claim and part-(3) of Theorem \ref{ti} to the
$\mathbb E$-isometry $(T_{\varphi},T_{\psi},T_z)$ to get
$T_{\varphi}=T_{\psi}^*T_z$ and $T_{\psi}=T_{\varphi}^*T_z$ and by
these two relations we have that
\[
\varphi(z)=G_1^*+G_2z \textup{ and } \psi(z)=G_2^*+G_1z \textup{
for some } G_1,G_2 \in\mathcal L(\mathcal D_{{\hat{T_3}}^*}).
\]
By the commutativity of $\varphi(z)$ and $\psi(z)$ we obtain
$$ [G_1,G_2]=0 \textup{ and } [G_1^*,G_1]=[G_2^*,G_2].$$
We now compute the fundamental operators of the $\mathbb
E$-co-isometry $(T_{\varphi}^*,T_{\psi}^*,T_z^*)$ that is of
$(T_{G_1^*+G_2z}^*,T_{G_2^*+G_1z}^*,T_z^*)$. Clearly $I-T_zT_z^*$
is the projection onto the space $\mathcal D_{T_z^*}$. Now
\[
T_{G_1^*+G_2z}^*- T_{G_2^*+G_1z}T_z^*=T_{G_1+G_2^* \bar z}-
T_{G_2^*+G_1z}T_{\bar z}= T_{G_1}=(I-T_zT_z^*)G_1(I-T_zT_z^*).
\]
Similarly,
\[
T_{G_2^*+G_1z}^*-T_{G_1^*+G_2z}T_z^* =(I-T_zT_z^*)G_2(I-T_zT_z^*).
\]
Therefore, $G_1,G_2$ are the fundamental operators of
$(T_{\varphi}^*,T_{\psi}^*,T_z^*)$. The fundamental operators of
$(\hat{T_1}^*,\hat{T_2}^*,\hat{T_3}^*)$ are $A_1,A_2$. Therefore
\[
\hat{T_1}^*-\hat{T_2}\hat{T_3}^*=D_{\hat{T_3}^*}A_1D_{\hat{T_3}^*}
\]
that is
\[
U^*(T_{\varphi}^*-T_{\psi}T_z^*)U=U^*D_{T_z^*}(UA_1U^*)D_{T_z}^*U
\]
or equivalently
\[
T_{\varphi}^*-T_{\psi}T_z^*=D_{T_z^*}(UA_1U^*)D_{T_z}^*.
\]
Similarly,
\[
T_{\psi}^*-T_{\varphi}T_z^*=D_{T_z^*}(UA_2U^*)D_{T_z}^*.
\]
Therefore, by the uniqueness of fundamental operators (see Theorem
\ref{funda-exist-unique}) we have that
\[
G_1=UA_1U^* \text{ and } G_2=UA_2U^*.
\]
From $[G_1,G_2]=0$ and $[G_1^*,G_1]=[G_2^*,G_2]$ it trivially
follows that $[A_1,A_2]=0$ and $[A_1^*,A_1]=[A_2^*,A_2]$. Also
since $(T_{\varphi},T_{\psi},T_z)$ is an $\mathbb E$-contraction,
we have that $\|T_{\varphi}\|\leq 1$ and hence
$\|\varphi(z)\|=\|G_1^*+G_2z\|\leq 1$ for all $z\in{\mathbb D}$.
Therefore, $\| A_1^*+A_2z\|=\|U^*(G_1^*+G_2)U \|\leq 1$ for all
$z\in{\mathbb D}$.\\

For the converse, we first prove that the triple of multiplication
operators $(M_{A_1^*+A_2z},M_{A_2^*+A_1z},M_z)$ on $L^2(E)$ is an
$\mathbb E$-unitary when $A_1,A_2$ satisfy the given conditions.
It is evident that $(M_{A_1^*+A_2z},M_{A_2^*+A_1z},M_z)$ is a
commuting triple of normal operators when $[A_1,A_2]=0$ and
$[A_1^*,A_1]=[A_2^*,A_2]$. Also
$M_{A_1^*+A_2z}=M_{A_2^*+A_1z}^*M_z$ and $M_z$ on $L^2(E)$ is
unitary. Therefore, by part-(3) of Theorem \ref{tu},
$(M_{A_1^*+A_2z},M_{A_2^*+A_1z},M_z)$ becomes an $\mathbb
E$-unitary if we prove that $\| M_{A_2^*+A_1z} \|\leq 1$ for all
$z\in\mathbb T$. We have that $\omega(A_1+zA_2)\leq 1$ for every
$z\in\mathbb T$, which is same as saying that
$\omega(z_1A_1+z_2A_2)\leq 1$ for all complex numbers $z_1,z_2$ of
unit modulus. Thus by Lemma \ref{basicnrlemma},
$$ (z_1A_1+z_2A_2)+(z_1A_1+z_2A_2)^*\leq 2I, $$
that is
$$ (z_1A_1+\bar{z_2}A_2^*)+(z_1A_1+\bar{z_2}A_2^*)^* \leq 2I.$$
Therefore, $\bar{z_2}(A_2^*+zA_1)+z_2(A_2^*+zA_1)^*\leq 2I$ for
all $z,z_2 \in\mathbb T$. This is same as saying that
$$ \textup{Re }z_2(A_2^*+zA_1)\leq I, \textup{ for all } z,z_2 \in\mathbb T. $$
Therefore, by Lemma \ref{basicnrlemma} again
$\omega(A_2^*+A_1z)\leq 1$ for any $z$ in $\mathbb T$. Since
$M_{A_2^*+A_1z}$ is a normal operator we have that $\|
M_{A_2^*+A_1z} \|=\omega(M_{A_2^*+A_1z})$ and thus $\|
M_{A_2^*+A_1z} \|$ for all $z\in\mathbb T$. Therefore,
$(M_{A_1^*+A_2z},M_{A_2^*+A_1z},M_z)$ on $L^2(E)$ is an $\mathbb
E$-unitary and hence $(T_{A_1^*+A_2z},T_{A_2^*+A_1z},T_z)$, being
the restriction of $(M_{A_1^*+A_2z},M_{A_2^*+A_1z},M_z)$ to the
common invariant subspace $H^2(E)$, is an $\mathbb E$-isometry.
Also $T_z$ on $H^2(E)$ is a pure isometry. Thus we conclude that
$(T_{A_1^*+A_2z},T_{A_2^*+A_1z},T_z)$ is a pure $\mathbb
E$-isometry.

\end{proof}

\section{A necessary condition for the existence of dilation}

Let us recall from section 1 the definitions of the $\mathbb
E$-isometric and $\mathbb E$-unitary dilations of an $\mathbb
E$-contraction. In fact they can be defined in a simpler way by
involving polynomials only. This is because the polynomials are
dense in the rational functions.
\begin{defn}
Let $(T_1,T_2,T_3)$ be a $\mathbb E$-contraction on $\mathcal H$.
A commuting tuple $(Q_1,Q_2,V)$ on $\mathcal K$ is said to be an
$\mathbb E$-isometric dilation of $(T_1,T_2,T_3)$ if $\mathcal H
\subseteq \mathcal K$, $(Q_1,Q_2,V)$ is an $\mathbb E$-isometry
and
$$ P_{\mathcal H}(Q_1^{m_1}Q_2^{m_2}V^n)|_{\mathcal H}=T_1^{m_1}T_2^{m_2}T_3^n,
\; \textup{ for all non-negative integers }m_1,m_2,n.
$$ Here $P_{\mathcal H}:\mathcal K \rightarrow \mathcal H$
is the orthogonal projection of $\mathcal K$ onto $\mathcal H$.
Moreover, the dilation is called {\em minimal} if
\[
\mathcal K=\overline{\textup{span}}\{ Q_1^{m_1}Q_2^{m_2}V^n h\,:\;
h\in\mathcal H \textup{ and }m_1,m_2,n\in \mathbb N \cup \{0\} \}.
\]
\end{defn}
\begin{defn}
A commuting tuple $(R_1,R_2,U)$ on $\mathcal K$ is said to be an
$\mathbb E$-unitary dilation of $(T_1,T_2,T_3)$ if $\mathcal H
\subseteq \mathcal K$, $(R_1,R_2,U)$ is an $\mathbb E$-unitary and
\[
P_{\mathcal H}(R_1^{m_1}R_2^{m_2}U^n)|_{\mathcal
H}=T_1^{m_1}T_2^{m_2}T_3^n, \; \textup{ for all non-negative
integers }m_1,m_2,n.
\]
Moreover, the dilation is called {\em
minimal} if
\[
\mathcal K=\overline{\textup{span}}\{ R_1^{m_1}R_2^{m_2}U^n h\,:\;
h\in\mathcal H \textup{ and }m_1,m_2,n\in \mathbb Z \}.
\]
Here $R_i^{m_i}={R_i^*}^{-m_i}$ for $i=1,2$ and $U^n={U^*}^{-n}$
when $m_i$ and $n$ are negative integers.
\end{defn}

\begin{prop}\label{exist-minimal}
If a $\mathbb E$-contraction $(T_1,T_2,T_3)$ defined on $\mathcal
H$ has a $\mathbb E$-isometric dilation, then it has a minimal
$E$-isometric dilation.
\end{prop}
\begin{proof}
Let $(Q_1,Q_2,V)$ on $\mathcal K\supseteq \mathcal H$ be a
$\mathbb E$-isometric dilation of $(T_1,T_2,T_3)$. Let $\mathcal
K_0$ be the space defined as
\[
\mathcal K_0=\overline{\textup{span}}\{ Q_1^{m_1}Q_2^{m_2}V^n
h\,:\; h\in\mathcal H \textup{ and }m_1,m_2,n\in \mathbb N \cup
\{0\} \}.
\]
Clearly $\mathcal K_0$ is invariant under $Q_1^{m_1},
Q_2^{m_2}$ and $V^n$, for any non-negative integer $m_1,m_2$ and
$n$. Therefore if we denote the restrictions of $Q_1,Q_2$ and $V$
to the common invariant subspace $\mathcal K_0$ by $Q_{11},
Q_{12}$ and $V_1$ respectively, we get
\[
Q_{11}^{m_1}k=Q_1^{m_1}k, \, Q_{12}^{m_2}k=Q_2^{m_2}k, \textup{
and } V_1^nk=V^nk,\quad \text{ for any }k\in\mathcal K_0.
\]
Hence
\[
\mathcal K_0=\overline{\textup{span}}\{
Q_{11}^{m_1}Q_{12}^{m_2}V_1^n h\,:\; h\in\mathcal H \textup{ and
}m_1,m_2,n\in \mathbb N \cup \{0\} \}.
\]
Therefore for any
non-negative integers $m_1,m_2$ and $n$ we have that
\[
P_{\mathcal
H}(Q_{11}^{m_1}Q_{12}^{m_2}V_1^{n})h=P_{\mathcal
H}(Q_1^{m_1}Q_2^{m_2}V^n )h, \quad \textup{ for all }h\in\mathcal
H .
\]
Now $(Q_{11},Q_{12},V_1)$ is an $\mathbb E$-contraction by being
the restriction of an $\mathbb E$-contraction $(Q_1,Q_2,V)$ to a
common invariant subspace $\mathcal K_0$. Also $V_1$, being the
restriction of an isometry to an invariant subspace, is also an
isometry. Therefore by Theorem \ref{ti} - part(2),
$(Q_{11},Q_{12},V_1)$ is an $\mathbb E$-isometry. Hence
$(Q_{11},Q_{12},V_1)$ is a minimal $\mathbb E$-isometric dilation
of $(T_1,T_2,T_3)$.

\end{proof}

\begin{prop}\label{dilation-extension}
Let $(Q_1,Q_2,V)$ on $\mathcal K$ be an $\mathbb E$-isometric
dilation of an $\mathbb E$-contraction $(T_1,T_2,T_3)$ on
$\mathcal H$. If $(Q_1,Q_2,V)$ is minimal, then
$(Q_1^*,Q_2^*,V^*)$ is an $\mathbb E$-co-isometric extension of
$(T_1^*,T_2^*,T_3^*)$.
\end{prop}
\begin{proof}
We first prove that $T_1P_{\mathcal H}=P_{\mathcal H}Q_1,
T_2P_{\mathcal H}=P_{\mathcal H}Q_2$ and $T_3P_{\mathcal
H}=P_{\mathcal H}V$. Clearly
\[
\mathcal K=\overline{\textup{span}}\{ Q_1^{m_1}Q_2^{m_2}V^n h\,:\;
h\in\mathcal H \textup{ and }m_1,m_2,n\in \mathbb N \cup \{0\} \}.
\]
Now for $h\in\mathcal H$ we have that
\begin{align*} T_1P_{\mathcal
H}(Q_1^{m_1}Q_2^{m_2}V^n h) =T_1(T_1^{m_1}T_2^{m_2}T_3^n h)
=T_1^{m_1+1}T_2^{m_2}T_3^n h & =P_{\mathcal
H}(Q_1^{m_1+1}Q_2^{m_2}V^n h)\\& =P_{\mathcal
H}Q_1(Q_1^{m_1}Q_2^{m_2}V^n h).
\end{align*}
Thus we have that $T_1P_{\mathcal H}=P_{\mathcal H}Q_1$ and
similarly we can prove that $T_2P_{\mathcal H}=P_{\mathcal H}Q_2$
and $T_3P_{\mathcal H}=P_{\mathcal H}V$. Also for $h\in\mathcal H$
and $k\in\mathcal K$ we have that
\begin{align*}
\langle T_1^*h,k \rangle
=\langle P_{\mathcal H}T_1^*h,k \rangle =\langle
T_1^*h,P_{\mathcal H}k \rangle  =\langle h,T_1P_{\mathcal H}k
\rangle =\langle h,P_{\mathcal H}Q_1k \rangle =\langle Q_1^*h,k
\rangle .
\end{align*}
Hence $T_1^*=Q_1^*|_{\mathcal H}$ and similarly
$T_2^*=Q_2^*|_{\mathcal H}$ and $T_3^*=V^*|_{\mathcal H}$.
Therefore, $(Q_1^*,Q_2^*,V^*)$ is an $\mathbb E$-co-isometric
extension of $(T_1^*,T_2^*,T_3^*)$.

\end{proof}

\begin{prop}\label{ultimate}
Let $\mathcal H_1$ be a Hilbert space and let $(T_1,T_2,T_3)$ be
an $\mathbb E$-contraction on $\mathcal H=\mathcal H_1\oplus
\mathcal H_1$ with fundamental operators $A_1,A_2$. Let
\begin{itemize}
 \item[(i)] $ Ker (D_{T_3})=\mathcal H_1\oplus \{0\} \textup{ and } \mathcal D_{T_3} =
 \{0\}\oplus \mathcal H_1\, ;$
\item[(ii)] $T_3(\mathcal D_{T_3})=\{0\}$ and $T_3
Ker(D_{T_3})\subseteq \mathcal D_{T_3}$.
\end{itemize}
If $(T_1^*,T_2^*,T_3^*)$ has an $\mathbb E$-isometric dilation
then
\begin{enumerate}
\item $A_1A_2=A_2A_1$, \item
$A_1^*A_1-A_1A_1^*=A_2^*A_2-A_2A_2^*$.
\end{enumerate}
\end{prop}
\begin{proof}
Let $(Q_1,Q_2,V)$ on a Hilbert space $\mathcal K$ be a minimal
$\mathbb E$-isometric dilation of $(T_1^*,T_2^*,T_3^*)$ (such a
minimal $\mathbb E$-isometric dilation exists by Proposition
\ref{exist-minimal}) so that $(Q_1^*,Q_2^*,V^*)$ is an $\mathbb
E$-co-isometric extension of $(T_1,T_2,T_3)$ by Proposition
\ref{dilation-extension}. Since $(Q_1,Q_2,V)$ on $\mathcal K$ is
an $E$-isometry, by part-(4) of Theorem \ref{ti}, $\mathcal K$ has
decomposition $\mathcal K=\mathcal K_1\oplus \mathcal K_2$ into
reducing subspaces $\mathcal K_1, \mathcal K_2$ of $Q_1,Q_2,V$
such that $(Q_1|_{\mathcal K_1},Q_2|_{\mathcal K_1},V|_{\mathcal
K_1})=(Q_{11},Q_{12},U_1)$ is an $\mathbb E$-unitary and
$(Q_1|_{\mathcal K_2},Q_2|_{\mathcal K_2},V|_{\mathcal
K_2})=(Q_{21},Q_{22},V_1)$ is a pure $\mathbb E$-isometry. Since
$(Q_{21},Q_{22},V_1)$ on $\mathcal K_2$ is a pure $\mathbb
E$-isometry, by Theorem \ref{model1}, $\mathcal K_2$ can be
identified with $H^2(E)$, where $E=\mathcal D_{V_1^*}$ and
$Q_{21},Q_{22},V_1$ can be identified with
$T_{\varphi},T_{\psi},T_z$ respectively on $H^2(E)$, where
$\varphi(z)=A+Bz$ and $\psi(z)=B^*+A^*z, \; z\in\mathbb D.$ Here
$A^*,B$ are the fundamental operators of
$(Q_{21}^*,Q_{22}^*,V_1^*)$. Again $H^2(E)$ can be identified with
$l^2(E)$ and $T_{\varphi},T_{\psi},T_z$ on $H^2(E)$ can be
identified with the multiplication operators
$M_{\varphi},M_{\psi},M_z$ on $l^2(E)$ respectively. So without
loss of generality we can assume that $K_2=l^2(E)$ and
$Q_{21}=M_{\varphi}, Q_{22}=M_{\psi}$ and $V_1=M_z$ on $l^2(E)$.
The block matrices of $M_{\varphi},M_{\psi},M_z$ are given by
\begin{align*}
M_{\varphi} &=\begin{bmatrix} A&0&0&\dots\\ B&A&0&\dots\\
0&B&A&\dots\\ \dots&\dots&\dots&\dots
\end{bmatrix},\; M_{\psi}=\begin{bmatrix} B^*&0&0&\dots \\ A^*&B^*&0&\dots\\
0&A^*&B^*&\dots\\\dots&\dots&\dots&\dots
 \end{bmatrix} \\ \textup{and }  M_z &=\begin{bmatrix} 0&0&0&\dots\\I&0&0&\dots\\0&I&0&\dots\\
 \dots&\dots&\dots&\dots& \end{bmatrix}.
\end{align*}
From now onward we shall consider $\mathcal H$ as a subspace of
$\mathcal K$ and $T_1,T_2,T_3$ on $\mathcal H$ as the restrictions
of $Q_1^*,Q_2^*,V^*$ respectively to $\mathcal H$.

\textit{Claim 1}. $\mathcal D_{T_3}\subseteq E \oplus \{0\}\oplus
\{0\}\oplus \cdots \subseteq l^2(E).$

\textit{Proof of claim.} Let $h=h_1\oplus h_2 \in \mathcal
D_{T_3}\subseteq \mathcal H$, where $h_1\in \mathcal K_1$ and
$h_2=(c_0,c_1,c_2,\dots)^T \in l^2(E)$. Here
$(c_0,c_1,c_2,\dots)^T$ denotes the transpose of the vector
$(c_0,c_1,c_2,\dots)$. Since $T_3(\mathcal D_{T_3})=\{0\}$, we
have that
\[
T_3h=V^*h=V^*(h_1\oplus h_2)=U_1^*h_1\oplus
M_z^*h_2=U_1^*h_1\oplus (c_1,c_2,\cdots)^T=0
\]
which implies that $ h_1=0 \textup{ and } c_1=c_2=\dots=0$.
This completes the proof of \textit{Claim 1}.\\

\textit{Claim 2}. $Ker(D_{T_3})\subseteq \{0\}\oplus E \oplus
\{0\}\oplus \{0\}\oplus \cdots \subseteq l^2(E).$

\textit{Proof of claim.} For $h=h_1\oplus h_2 \in
Ker(D_{T_3})\subseteq \mathcal H$, where $h_1\in \mathcal K_1$ and
$h_2=(c_0,c_1,c_2,\dots)^T \in l^2(E)$, we have that
\[
D_{T_3}^2h=(I-{T_3}^*{T_3})h=P_{\mathcal H}(I-VV^*)h=P_{\mathcal
H}(h_1\oplus h_2 - h_1 \oplus M_zM_z^*h_2)=0
\]
which implies that
$ P_{\mathcal H}(h_1\oplus h_2)=P_{\mathcal H}(h_1\oplus
M_zM_z^*h_2)$. Therefore,
\[
h_1\oplus (c_0,c_1,\cdots)^T = P_{\mathcal H}(h_1\oplus
(0,c_1,c_2,\cdots)^T)
\]
which further implies that $ \|h_1\oplus
(0,c_1,c_2,\cdots)^T\|\geq \|h_1 \oplus (c_0,c_1,c_2,\cdots)^T\|$.
Thus $c_0=0$. Again $T_3(Ker (D_{T_3}))\subseteq \mathcal
D_{T_3}$. Therefore, for $h_1\oplus (0,c_1,c_2,\dots)^T \in
Ker(D_{T_3})$, we have that
\[
T_3(h_1\oplus (0,c_1,c_2,\dots)^T)=U_1^*h_1 \oplus
M_z^*(0,c_1,c_2,\cdots)^T =U_1^*h_1 \oplus (c_1,c_2,\cdots)^T\in
\mathcal D_{T_3}.
\]
Then by Claim 1, $h_1=0$ and
$c_2=c_3=\dots=0$. Hence \textit{Claim 2} is established.\\

Now since $\mathcal H=\mathcal D_{T_3} \oplus Ker(D_{T_3})$, we
can conclude that $\mathcal H \subseteq E\oplus E\oplus
\{0\}\oplus \{0\}\oplus \cdots \subseteq l^2(E)=\mathcal K_2$.
Therefore, $(M_{\varphi}^*,M_{\psi}^*,M_z^*)$ on $l^2(E)$ is an
$\mathbb E$-co-isometric extension of $(T_1,T_2,T_3)$. We now
compute the fundamental operators of
$(M_{\varphi}^*,M_{\psi}^*,M_z^*)$.
\begin{align*}
& M_{\varphi}^*-M_{\psi}M_z^*
\\& =\begin{bmatrix} A^*&B^*&0&\cdots\\ 0&A^*&B^*&\cdots\\
0&0&A^*&\cdots\\ \vdots&\vdots&\vdots&\ddots
\end{bmatrix} - \begin{bmatrix} B^*&0&0&\dots\\ A^*&B^*&0&\cdots\\ 0&A^*&B^*&\cdots\\
\vdots&\vdots&\vdots&\ddots \end{bmatrix}
\begin{bmatrix} 0&I&0&\cdots \\ 0&0&I&\cdots \\ 0&0&0&\cdots \\ \vdots&\vdots&\vdots&\ddots
\end{bmatrix} \\& = \begin{bmatrix} A^*&B^*&0&\cdots\\ 0&A^*&B^*&\cdots\\
0&0&A^*&\cdots\\ \vdots&\vdots&\vdots&\ddots
\end{bmatrix}
- \begin{bmatrix} 0&B^*&0&\cdots\\ 0&A^*&B^*&\cdots\\
0&0&A^*&\cdots\\ \vdots&\vdots&\vdots&\ddots
\end{bmatrix} \\& = \begin{bmatrix} A^*&0&0&\cdots \\ 0&0&0&\cdots \\ 0&0&0&\cdots \\
\vdots&\vdots&\vdots&\ddots \end{bmatrix}.
\end{align*}

Similarly $$ M_{\psi}^* - M_{\varphi}M_z^* = \begin{bmatrix} B&0&0&\cdots \\ 0&0&0&\cdots \\ 0&0&0&\cdots \\
\vdots&\vdots&\vdots&\ddots \end{bmatrix}. $$
Also
\begin{align*} D_{M_z^*}^2 &=I-M_zM_z^* \\
\\& = \begin{bmatrix} I&0&0&\cdots \\ 0&0&0&\cdots\\
0&0&0&\cdots \\ \vdots&\vdots&\vdots&\ddots
\end{bmatrix} .
\end{align*}

Therefore, $\mathcal D_{M_z^*}=E\oplus \{0\}\oplus \{0\}\cdots$
and $D_{M_z^*}^2=D_{M_z^*}=I_d$ on $E\oplus \{0\}\oplus
\{0\}\cdots$. If we set
\begin{align}\label{011}
\hat{A_1}=\begin{bmatrix} A^*&0&0&\dots\\ 0&0&0&\dots\\
0&0&0&\dots\\ \dots&\dots&\dots&\dots
\end{bmatrix},\; \hat{A_2}=\begin{bmatrix} B&0&0&\dots \\ 0&0&0&\dots\\
0&0&0&\dots\\\dots&\dots&\dots&\dots
\end{bmatrix},
\end{align}
then
$$ M_{\varphi}^*-M_{\psi}M_z^*=D_{M_z^*} \hat{A_1} D_{M_z^*} \textup{ and }
M_{\psi}^*-M_{\varphi}M_z^*=D_{M_z^*} \hat{A_2} D_{M_z^*}.  $$

Therefore, $\hat{A_1},\hat{A_2}$ are the fundamental operators of
$(M_{\varphi}^*,M_{\psi}^*,M_z^*)$.

 Let us denote $(M_{\varphi}^*,M_{\psi}^*,M_z^*)$ by
$(R_1,R_2,W)$. Therefore,
 \begin{eqnarray}
 \label{01}&R_1-R_2^*W=D_W\hat{A_1}D_W
\\ \label{02}& R_2-R_1^*W=D_W\hat{A_2}D_W.
 \end{eqnarray}

\noindent \textit{Claim 3.} $\hat{A_i}D_W|_{\mathcal D_{T_3}}
\subseteq \mathcal D_{T_3}$ and $\hat{A_i}^*D_W|_{\mathcal
D_{T_3}} \subseteq \mathcal D_{T_3}$
for $i=1,2$.\\
\textit{Proof of claim.} Clearly $D_W=D_{M_z^*}=I_d$ on $\mathcal
D_W$. Let $h_0=(c_0,0,0,\cdots)^T\in\mathcal D_{T_3}$. Then
$\hat{A_1}D_Wh_0=(A^*c_0,0,0,\cdots)^T=M_{\varphi}^*h_0=R_1h_0$.
Since $R_1|_{\mathcal H}=S_1$, $R_1h_0 \in \mathcal H$. Therefore
$(A^*c_0,0,0,\cdots)^T \in \mathcal D_{T_3}$ and
$\hat{A_1}D_W|_{\mathcal D_{T_3}} \subseteq \mathcal D_{T_3}$.
Similarly we can prove that $\hat{A_2}D_W|_{\mathcal D_{T_3}}
\subseteq \mathcal D_{T_3}$.

We compute the adjoint of $T_3$. Let $(c_0,c_1,0,\cdots)^T$ and
$(d_0,d_1,0,\cdots)^T$ be two arbitrary elements in $\mathcal H$
where $(c_0,0,0,\cdots)^T, (d_0,0,0,\cdots)^T \in \mathcal
D_{T_3}$ and $(0,c_1,0,\cdots)^T,(0,d_1,0,\cdots)^T \in
Ker(D_{T_3})$. Now
\begin{align*}
\langle T_3^*(c_0,c_1,0,\cdots)^T,(d_0,d_1,0,\cdots)^T \rangle &
=\langle (c_0,c_1,0,\cdots)^T,T_3(d_0,d_1,0,\cdots)^T \rangle \\&
=\langle (c_0,c_1,0,\cdots)^T,W(d_0,d_1,0,\cdots)^T \rangle \\&
=\langle (c_0,c_1,0,\cdots)^T,(d_1,0,0,\cdots)^T \rangle \\&
=\langle c_0,d_1 \rangle_E \\& =\langle
(0,c_0,0,\cdots)^T,(d_0,d_1,0,\cdots)^T \rangle.
\end{align*}
Therefore $$ T_3^*(c_0,c_1,0,\cdots)^T=(0,c_0,0,\cdots)^T.$$ Now
$h_0=(c_0,0,0,\cdots)^T\in\mathcal D_{T_3}$ implies that
$T_3^*h_0=(0,c_0,0,\cdots)^T \in \mathcal H$ and
$M_{\psi}^*(0,c_0,0,\cdots)^T=R_2(0,c_0,0,\cdots)^T =
(Ac_0,0,0,\cdots)^T \in \mathcal H$. In particular,
$(Ac_0,0,0,\cdots)^T \in \mathcal D_{T_3}$. Therefore
$\hat{A_1}^*D_Wh_0=(Ac_0,0,0,\cdots)^T \in \mathcal D_{T_3}$ and
$\hat{A_2}^*D_W|_{\mathcal D_{T_3}} \subseteq \mathcal D_{T_3}$.
Similarly we can prove that $\hat{A_2}^*D_W|_{\mathcal
D_{T_3}}\subseteq
\mathcal D_{T_3}$. Hence \textit{Claim 3} is proved.\\

\noindent \textit{Claim 4.} $\hat{A_i}|_{\mathcal D_{T_3}}=A_i$ and $\hat{A_i}^*|_{\mathcal D_{T_3}}=A_i^*$ for $i=1,2$.\\
\textit{Proof of Claim.}
 It is obvious that $\mathcal D_{T_3} \subseteq \mathcal D_W =E\oplus\{0\}\oplus \{0\}\oplus
 \cdots$. Now since $W|_{\mathcal H}=T_3$ and $D_W$ is projection onto $\mathcal
 D_W$, we have that $ D_W|_{\mathcal H}= D_W^2|_{\mathcal H}=D_W^2|_{\mathcal D_{T_3}} =D_{T_3}^2
 $. Therefore, $D_{T_3}^2$ is a projection onto $\mathcal D_{T_3}$
 and $D_{T_3}^2=D_{T_3}$.
 From (\ref{01}) we have that
 \begin{eqnarray} \label{03} P_{\mathcal H}(R_1-R_2^*W)|_{\mathcal H}=
 P_{\mathcal H}(D_{W}\hat{A_1}D_{W})|_{\mathcal H}.
 \end{eqnarray}
 Since $(R_1,R_2,W)$ is an $\mathbb E$-co-isometric extension of
 $(T_1,T_2,T_3)$, the LHS of (\ref{03}) is equal to
 $T_1-T_2^*T_3$.
Again since $A_1,A_2$ are the fundamental operators of
$(T_1,T_2,T_3)$, we have that
 \begin{eqnarray} \label{04}
 T_1-T_2^*T_3=D_{T_3}A_1D_{T_3}, \quad A_1\in \mathcal L(\mathcal D_{T_3}).
 \end{eqnarray}
It is clear that $T_1-T_2^*T_3$ is $0$ on the ortho-complement of
$\mathcal D_{T_3}$, that is on $Ker(D_{T_3})$. Therefore,
\begin{eqnarray} \label{05} T_1-T_2^*T_3=P_{\mathcal D_{T_3}}(R_1-R_2^*W)|_{\mathcal D_{T_3}}=
 P_{\mathcal D_{T_3}}(D_{W}\hat{A_1}D_{W})|_{\mathcal D_{T_3}}.
 \end{eqnarray}
Again since $D_W|_{\mathcal D_{T_3}}=D_{T_3}=I_d$ on $\mathcal
D_{T_3}$, the RHS of (\ref{05}) is equal to\\
$(D_{W}\hat{A_1}D_{W})|_{\mathcal D_{T_3}}$ and hence
\begin{eqnarray} \label{06} T_1-T_2^*T_3=(R_1-R_2^*W)|_{\mathcal D_{T_3}}=
 (D_{W}\hat{A_1}D_{W})|_{\mathcal D_{T_3}}=D_{T_3}\hat{A_1}D_{T_3}.
 \end{eqnarray}
The last identity follows from the fact (\textit{Claim 3}) that
$\hat{A_1}D_W|_{\mathcal D_{T_3}}\subseteq \mathcal D_{T_3}$. By
the uniqueness of $A_1$ we get that $\hat{A_1}|_{\mathcal
D_{T_3}}=A_1$. Also since $\mathcal D_{T_3}$ is invariant under
$\hat{A_1}^*$ by \textit{Claim 3}, we have that
$\hat{A_1}^*|_{\mathcal D_{T_3}}=A_1^*$. Similarly we can prove
that $\hat{A_2}|_{\mathcal D_{T_3}}=A_2$ and
$\hat{A_2}^*|_{\mathcal D_{T_3}}=A_2^*$. Thus the proof to \textit{Claim 4} is complete.\\

Now since $(M_{\varphi},M_{\psi},M_z)$ on $l^2(E)$ is an $\mathbb
E$-isometry,
 $M_{\varphi}$ and $M_{\psi}$ commute, that is
\begin{align*}
& \begin{bmatrix} A&0&0&\dots\\ B&A&0&\dots\\
0&B&A&\dots\\ \dots&\dots&\dots&\dots
\end{bmatrix} \begin{bmatrix} B^*&0&0&\dots \\ A^*&B^*&0&\dots\\
0&A^*&B^*&\dots\\\dots&\dots&\dots&\dots
 \end{bmatrix} \\ = & \begin{bmatrix} B^*&0&0&\dots \\ A^*&B^*&0&\dots\\
0&A^*&B^*&\dots\\\dots&\dots&\dots&\dots\end{bmatrix}\begin{bmatrix} A&0&0&\dots\\ B&A&0&\dots\\
0&B&A&\dots\\ \dots&\dots&\dots&\dots \end{bmatrix}
\end{align*}
which implies that
\begin{align*}
& \begin{bmatrix} AB^*&0&0&\dots\\ BB^*+AA^*&AB^*&0&\dots\\
BA^*&BB^*+AA^*&AB^*&\dots\\ \dots&\dots&\dots&\dots
\end{bmatrix} \\ = & \begin{bmatrix} B^*A&0&0&\dots \\ A^*A+B^*B&B^*A&0&\dots\\
A^*B&A^*A+B^*B&B^*A&\dots\\\dots&\dots&\dots&\dots\end{bmatrix}.
\end{align*}
Comparing both sides we obtain the following,
\begin{enumerate}
\item $A^*B=BA^*$ \item $A^*A-AA^*=BB^*-B^*B$. \end{enumerate}
Therefore from (\ref{011}) we have that
\begin{enumerate}
\item $\hat{A_1}\hat{A_2}=\hat{A_2}\hat{A_1}$ \item
$\hat{A_1}^*\hat{A_1}-\hat{A_1}\hat{A_1}^*=\hat{A_2}^*\hat{A_2}-\hat{A_2}\hat{A_2}^*$.
\end{enumerate}
Taking restriction of the above two operator identities to the
subspace $\mathcal D_{T_3}$ we get
\begin{enumerate}
\item $A_1A_2=A_2A_1$ \item $A_1^*A_1-A_1A_1^*=A_2^*A_2-A_2A_2^*$.
\end{enumerate} The proof is now complete.

\end{proof}

\section{A counter example}

Let $\mathcal H_1=l^2(E)\oplus l^2(E),\; E=\mathbb C^2$ and let
$\mathcal H=\mathcal H_1 \oplus \mathcal H_1$. Let $T_1,T_2,T_3$
on $\mathcal H_1\oplus \mathcal H_1$ be the block operator
matrices
\[
T_1=\begin{bmatrix} 0&0\\0&J \end{bmatrix},\,
T_2=\begin{bmatrix} 0&0\\0&0
\end{bmatrix} \text{ and }T_3=\begin{bmatrix} 0&0\\Y&0
\end{bmatrix}
\]
where
\[
J=\begin{bmatrix} F&0\\0&0 \end{bmatrix} \text{ and }
Y=\begin{bmatrix} 0&V\\I&0 \end{bmatrix} \text{ on } \mathcal
H_1=l^2(E)\oplus l^2(E).
\]
Here $V=M_z$ and $I=I_d$ on $l^2(E)$
and $F$ on $l^2(E)$ is defined as
\begin{align*}
F\;:&\;l^2(E)\rightarrow l^2(E) \\& (c_0,c_1,c_2,\cdots)^T \mapsto
(F_1c_0,0,0,\cdots)^T,
\end{align*}
where we choose
\[
F_1=\begin{pmatrix} 0&\frac{1}{4} \\ 0&0 \end{pmatrix}
\]
so that $F_1$ is a non-normal contraction such that $F_1^2=0$.
Clearly $F^2=0$ and $F^*F\neq FF^*$. Since $FV=0$, $JY=0$ and thus
the product of any two of $T_1,T_2,T_3$ is equal to $0$. Now we
unfold the operators $T_1,T_2,T_3$ and write their block matrices
with respect to the decomposition $\mathcal H = l^2(E)\oplus
l^2(E)\oplus l^2(E)\oplus l^2(E)$:
\[
T_1=\begin{bmatrix} 0&0&0&0\\0&0&0&0\\0&0&F&0\\0&0&0&0
\end{bmatrix},\; T_2=
\begin{bmatrix} 0&0&0&0\\0&0&0&0\\0&0&0\emph{}&0\\0&0&0&0
\end{bmatrix} \textup{ and } T_3=\begin{bmatrix}
 0&0&0&0\\0&0&0&0\\0&V&0&0\\I&0&0&0 \end{bmatrix}.
\]
We shall prove later that $(T_1,T_2,T_3)$ is an $\mathbb
E$-contraction and let us assume it for now. Here
\begin{align*}
 D_{T_3}^2=I-T_3^*T_3 &=\begin{bmatrix} I&0&0&0\\0&I&0&0\\0&0&I&0\\0&0&0&I \end{bmatrix}-
 \begin{bmatrix} 0&0&0&I\\0&0&V^*&0\\0&0&0&0\\0&0&0&0 \end{bmatrix}
 \begin{bmatrix} 0&0&0&0\\0&0&0&0\\0&V&0&0\\I&0&0&0 \end{bmatrix} \\ &=
 \begin{bmatrix} 0&0&0&0\\0&0&0&0\\0&0&I&0\\0&0&0&I \end{bmatrix}=D_{T_3}.
 \end{align*}
Clearly $\mathcal D_{T_3}=\{0\}\oplus \{0\}\oplus l^2(E)\oplus
l^2(E)=\{0\}\oplus \mathcal H_1$ and $Ker(D_{T_3})=l^2(E)\oplus
l^2(E)\oplus \{0\}\oplus \{0\}=\mathcal H_1 \oplus \{0\}$. Also
for a vector $k_0=(h_0,h_1,0,0)^T\in Ker(D_{T_3})$ and for a
vector $k_1=(0,0,h_2,h_3)^T\in \mathcal D_{T_3}$,
$$T_3k_0=\begin{bmatrix} 0&0&0&0\\0&0&0&0\\0&V&0&0\\I&0&0&0 \end{bmatrix}
(h_0,h_1,0,0)^T=(0,0,Vh_1,h_0)^T\in \mathcal D_{T_3} $$ and
$$T_3k_1=\begin{bmatrix} 0&0&0&0\\0&0&0&0\\0&V&0&0\\I&0&0&0 \end{bmatrix}
(0,0,h_2,h_3)^T = (0,0,0,0)^T.$$

Thus $(T_1,T_2,T_3)$ satisfies all the conditions of Proposition
\ref{ultimate}. We now compute the fundamental operators $A_1,A_2$
of $(T_1,T_2,T_3)$.
\[
T_1-T_2^*T_3=T_1=\begin{bmatrix}
0&0&0&0\\0&0&0&0\\0&0&F&0\\0&0&0&0 \end{bmatrix}
=D_{T_3}A_1D_{T_3}=\begin{bmatrix}
0&0&0&0\\0&0&0&0\\0&0&I&0\\0&0&0&I
\end{bmatrix}A_1\begin{bmatrix} 0&0&0&0\\0&0&0&0\\0&0&I&0\\0&0&0&I \end{bmatrix}.
\]
Since $\mathcal D_{T_3}=\{0\}\oplus \mathcal H_1$ and
$A_1\in\mathcal L(\mathcal D_{T_3})$ we can set
\[
A_1=0\oplus
\begin{bmatrix} F&0\\0&0 \end{bmatrix} \quad \text{ on } \{0\}\oplus \mathcal H_1 (= \mathcal
D_{T_3})
\]
so that
\[
T_1-T_2^*T_3=D_{T_3}A_1D_{T_3}.
\]
Again $T_1^*T_3=0$ as $X^*V=0$ and therefore $T_2-T_1^*T_3=0$.
This shows that the fundamental operator $A_2$, for which
$T_2-T_1^*T_3=D_{T_3}A_2D_{T_3}$ holds, has to be equal to $0$.
Clearly
\[
A_1^*A_1-A_1A_1^*=0\oplus \begin{bmatrix}
F^*F-FF^*&0\\0&0
\end{bmatrix} \neq 0 \textup{ as } F^*F\neq FF^*
\]
but $A_2^*A_2-A_2A_2^*=0$. This violets the conclusion of
Proposition \ref{ultimate} and it is guaranteed that the $\mathbb
E$-contraction $(T_1^*,T_2^*,T_3^*)$ does not have an $\mathbb
E$-isometric dilation. Since every $\mathbb E$-unitary dilation is
necessarily an $\mathbb E$-isometric dilation,
$(T_1^*,T_2^*,T_3^*)$ does not have an $\mathbb E$-unitary dilation.\\

Now we prove that $(T_1,T_2,T_3)$ is an $\mathbb E$-contraction.
By Lemma \ref{simpler}, it suffices to show that $
\|p(T_1,T_2,T_3)\|\leq \|p\|_{\infty, \overline{\mathbb E}}\;, $
for any polynomial $p(x_1,x_2,x_3)$ in the co-ordinates of
$\mathbb E$. Let
\[
p(x_1,x_2,x_3)=a_0+\displaystyle \sum_{i=1}^3a_ix_i+
q(x_1,x_2,x_3),
\]
where $q$ is a polynomial containing only terms
of second or higher degree. Now
\begin{align*}
p(T_1,T_2,T_3)=a_0I+a_1T_1+a_3T_3=\begin{bmatrix} a_0I&0\\a_3Y &
a_0I+a_1J
\end{bmatrix}
\end{align*}
Since $Y$ is a contraction and $\|J\|=\dfrac{1}{4}$, it is obvious
that
\[
\left\| \begin{bmatrix} a_0I&0\\a_3Y & a_0I+a_1J
\end{bmatrix} \right\| \leq \left\| \begin{pmatrix} |a_0|&0\\
|a_3|&|a_0|+\dfrac{|a_1|}{4}
\end{pmatrix} \right\|.
\]
We divide the rest of the proof into two cases.\\

\textbf{Case 1.} When $|a_0|\leq |a_1|$.\\

We show that
\[
\left\| \begin{pmatrix} |a_0|&0\\
|a_3|&|a_0|+\dfrac{|a_1|}{4}
\end{pmatrix} \right\| \leq
\left\| \begin{pmatrix} |a_0|&0\\
|a_1|+|a_3|&|a_0|
\end{pmatrix} \right\|\,.
\]
Let $\begin{pmatrix} \epsilon \\ \delta
\end{pmatrix}$ be a unit vector in $\mathbb C^2$ such that
\[
\left\| \begin{pmatrix} |a_0|&0\\
|a_3|&|a_0|+\dfrac{|a_1|}{4} \end{pmatrix} \right\|= \left\| \begin{pmatrix} |a_0|&0\\
|a_3|&|a_0|+\dfrac{|a_1|}{4} \end{pmatrix}\begin{pmatrix} \epsilon
\\ \delta
\end{pmatrix} \right\|.
\]
Without loss of generality we can choose $\epsilon, \delta \geq 0$
because
\[
\left\| \begin{pmatrix} |a_0|&0\\
|a_3|&|a_0|+\dfrac{|a_1|}{4} \end{pmatrix}\begin{pmatrix} \epsilon
\\ \delta
\end{pmatrix} \right\|^2
 =|a_0\epsilon|^2+\left| |a_3\epsilon|+\left( |a_0|+\dfrac{|a_0|}{4} \right)\delta\right|^2
\]
and if we replace $\begin{pmatrix} \epsilon \\ \delta
\end{pmatrix}$ by $\begin{pmatrix} |\epsilon | \\ |\delta|
\end{pmatrix}$ we see that
\[
\left\| \begin{pmatrix} |a_0|&0\\
|a_3|&|a_0|+\dfrac{|a_1|}{4} \end{pmatrix}\begin{pmatrix}
|\epsilon|
\\ |\delta|
\end{pmatrix} \right\|^2 \geq
\left\| \begin{pmatrix} |a_0|&0\\
|a_3|&|a_0|+\dfrac{|a_1|}{4} \end{pmatrix}\begin{pmatrix} \epsilon
\\ \delta
\end{pmatrix} \right\|^2 \,.
\]
So, assuming $\epsilon, \delta \geq 0$ we get
\begin{align}
&\left\| \begin{pmatrix} |a_0|&0\\
|a_3|&|a_0|+\dfrac{|a_1|}{4} \end{pmatrix}\begin{pmatrix} \epsilon
\\ \delta
\end{pmatrix} \right\|^2 \notag \\
&=|a_0\epsilon|^2+\left\{ |a_3\epsilon|+\left(
|a_0|+\dfrac{|a_1|}{4} \right)\delta \right\} ^2 \notag \\
&=|a_0\epsilon|^2+|a_3\epsilon|^2+ \left\{
|a_0|^2+\dfrac{|a_0a_1|}{2}+\dfrac{|a_1|^2}{16} \right\}{\delta}^2
+
2|a_3|\left( |a_0|+\dfrac{|a_1|}{4} \right)\epsilon\delta \notag \\
&= \left\{
(|a_0|^2+|a_3|^2)\epsilon^2+|a_0|^2\delta^2+2|a_0a_3|\epsilon\delta
\right\} + \left\{ \dfrac{|a_1|^2}{16}+\dfrac{|a_0a_1|}{2}
\right\}\delta^2 + \dfrac{|a_1a_3|}{2}\epsilon\delta \,.
\label{eqn:ex1}
\end{align}
Again
\begin{align}
& \left\| \begin{pmatrix} |a_0|&0\\
|a_1|+|a_3|&|a_0| \end{pmatrix}\begin{pmatrix} \epsilon \\ \delta
\end{pmatrix} \right\|^2 \notag \\
& = |a_0\epsilon|^2+\{ (|a_1|+|a_3|)\epsilon+|a_0|\delta \}^2
\notag \\
& = |a_0|^2\epsilon^2+\{ |a_1|^2+|a_3|^2+2|a_1a_3|
\}\epsilon^2+2|a_0|(|a_1|+|a_3|)\epsilon\delta+|a_0|^2\delta^2
\notag \\
& =\left\{
(|a_0|^2+|a_3|^2)\epsilon^2+|a_0|^2\delta^2+2|a_0a_3|\epsilon\delta
\right\} +
(|a_1|^2\epsilon^2+2|a_0a_1|\epsilon\delta)+2|a_1a_3|\epsilon^2
\,. \label{eqn:ex2}
\end{align}
We now compare (\ref{eqn:ex1}) and (\ref{eqn:ex2}). If $\epsilon
\geq \delta$ then
\[
(|a_1|^2\epsilon^2+2|a_0a_1|\epsilon\delta)+2|a_1a_3|\epsilon^2
\geq \left( \dfrac{|a_1|^2}{16}+\dfrac{|a_0a_1|}{2}
\right)\delta^2+\frac{|a_1a_3|}{2}\epsilon\delta
\]
Therefore, it is evident from (\ref{eqn:ex1}) and (\ref{eqn:ex2})
that
\[
\left\| \begin{pmatrix} |a_0|&0\\
|a_3|&|a_0|+\dfrac{|a_1|}{4}
\end{pmatrix}\begin{pmatrix} \epsilon \\ \delta
\end{pmatrix} \right\|^2 \leq
\left\| \begin{pmatrix} |a_0|&0\\
|a_1|+|a_3|&|a_0|
\end{pmatrix}\begin{pmatrix} \epsilon \\ \delta
\end{pmatrix} \right\|^2\,.
\]
If $\epsilon < \delta$ we consider the unit vector $\begin{pmatrix} \delta \\
\epsilon \end{pmatrix}$ and it suffices if we show that
\[
\left\| \begin{pmatrix} |a_0|&0\\
|a_3|&|a_0|+\dfrac{|a_1|}{4}
\end{pmatrix}\begin{pmatrix} \epsilon \\ \delta
\end{pmatrix} \right\|^2 \leq
\left\| \begin{pmatrix} |a_0|&0\\
|a_1|+|a_3|&|a_0|
\end{pmatrix}\begin{pmatrix} \delta \\ \epsilon
\end{pmatrix} \right\|^2\,.
\]
A computation similar to (\ref{eqn:ex2}) gives
\begin{align}
& \left\| \begin{pmatrix} |a_0|&0\\
|a_1|+|a_3|&|a_0|
\end{pmatrix}\begin{pmatrix} \delta \\ \epsilon
\end{pmatrix} \right\|^2 \notag \\
& = |a_0|^2\delta^2+\{ |a_1|^2+|a_3|^2+2|a_1a_3|
\}\delta^2+2|a_0|(|a_1|+|a_3|)\epsilon\delta +|a_0|^2\epsilon^2
\notag \\
&= \{ |a_0|^2(\epsilon^2+\delta^2)+2|a_0a_3|\epsilon\delta \}+\{
|a_1|^2+|a_3|^2+2|a_1a_3| \}\delta^2+2|a_0a_1|\epsilon\delta  \notag \\
& = \{ |a_0|^2+2|a_0a_3|\epsilon\delta \}+\{
|a_1|^2+|a_3|^2+2|a_1a_3| \}\delta^2+2|a_0a_1|\epsilon\delta \,.
 \label{eqn:ex3}
\end{align}
In the last equality we used the fact that
$|\epsilon|^2+|\delta|^2=1$. Again from (\ref{eqn:ex1}) we have
\begin{align}
& \left\| \begin{pmatrix} |a_0|&0\\
|a_3|&|a_0|+\dfrac{|a_1|}{4} \end{pmatrix}\begin{pmatrix} \epsilon
\\ \delta
\end{pmatrix} \right\|^2 \notag \\
& = \{ |a_0|^2(\epsilon^2+\delta^2)+2|a_0a_3|\epsilon\delta \} +
\left\{ |a_3|^2\epsilon^2+\dfrac{|a_1a_3|}{2}\epsilon\delta
\right\}
+\left\{ \dfrac{|a_1|^2}{16}+\dfrac{|a_0a_1|}{2} \right\}\delta^2 \notag \\
& \leq \{ |a_0|^2(\epsilon^2+\delta^2)+2|a_0a_3|\epsilon\delta \}
+ \left\{ |a_3|^2\epsilon^2+\dfrac{|a_1a_3|}{2}\epsilon\delta
\right\} + \left\{ \dfrac{|a_1|^2}{16}+\dfrac{|a_1|^2}{2}
\right\}\delta^2 \notag \\
&=\{ |a_0|^2+2|a_0a_3|\epsilon\delta \} + \left\{
\dfrac{9|a_1|^2}{16}\delta^2+|a_3|^2\epsilon^2+\dfrac{|a_1a_3|}{2}\epsilon\delta
\right\} \label{eqn:ex4}
\end{align}
The last inequality follows from the fact that $|a_0|\leq |a_1|$.
Since $\epsilon<\delta$ we can conclude from (\ref{eqn:ex3}) and
(\ref{eqn:ex4}) that
\[
\left\| \begin{pmatrix} |a_0|&0\\
|a_3|&|a_0|+\dfrac{|a_1|}{4}
\end{pmatrix}\begin{pmatrix} \epsilon \\ \delta
\end{pmatrix} \right\|^2 \leq
\left\| \begin{pmatrix} |a_0|&0\\
|a_1|+|a_3|&|a_0|
\end{pmatrix}\begin{pmatrix} \delta \\ \epsilon
\end{pmatrix} \right\|^2\,.
\]
Therefore,
\[
\left\| p(T_1,T_2,T_3) \right\|\leq \left\| \begin{pmatrix} |a_0|&0\\
|a_3|&|a_0|+\dfrac{|a_1|}{4} \end{pmatrix} \right\| \leq \left\| \begin{pmatrix} |a_0|&0\\
|a_1|+|a_3|&|a_0| \end{pmatrix} \right\|.
\]
A classical result of
Caratheodory and Fej\'{e}r states that
\[
\inf \, \left\|b_0+b_1z+r(z) \right\|_{\infty, \overline{\mathbb
D}} =
\left\| \begin{pmatrix} b_0&0\\
b_1&b_0 \end{pmatrix} \right\|,
\]
where the infimum is taken over all polynomials $r(z)$ in one
variable which contain only terms of degree two or higher. For an
elegant proof to this result, see Sarason's seminal paper
\cite{sarason}, where the result is derived as a consequence of
the classical commutant lifting theorem of Sz.-Nagy and Foias (see
\cite{nagy}). Using this fact we have that
\begin{align}
\left\| p(T_1,T_2,T_3) \right\| & \nonumber \leq \left\| \begin{pmatrix} |a_0|&0\\
|a_1|+|a_3|&|a_0| \end{pmatrix} \right\| \\& \nonumber = \inf \,
\||a_0|+(|a_1|+|a_3|)z+r(z)\|_{\infty, \overline{\mathbb D}} \\&
\label{014} \leq
 \inf \, \| |a_0|+|a_1|x_1+|a_3|x_3+r_1(x_1,x_2,x_3) \|_{\infty,
 \Lambda}\\& \label{015} \leq \inf \, \| |a_0|+|a_2|+|a_1|x_1+|a_3|x_3+r_1(x_1,x_2,x_3)\|_{\infty, \Lambda}
 \\& \nonumber
 = \inf \, \| |a_0|+|a_1|x_1+|a_2|x_2+|a_3|x_3+r_1(x_1,x_2,x_3) \|_{\infty, \Lambda}\\&
\label{016} \leq \| a_0+a_1x_1+a_2x_2+a_3x_3+q(x_1,x_2,x_3)
\|_{\infty,
 \Lambda} \\& \nonumber \leq \| a_0+a_1x_1+a_2x_2+a_3x_3+q(x_1,x_2,x_3)
\|_{\infty,
 \overline{\mathbb E}} \\& \nonumber =\|p(x_1,x_2,x_3) \|_{\infty, \overline{\mathbb E}}.
\end{align}
Here $\Lambda=\{ (x,1,x)\,:\,x\in \overline{\mathbb D} \}\subseteq
\overline{\mathbb E}$ (by choosing $\beta_1=0,\, \beta_2=1$ in
Theorem \ref{thm1}) and $r(z)$ and $r_1(x_1,x_2,x_3)$ range over
polynomials of degree two or higher. The inequality (\ref{014})
was obtained by putting $x_1=x_3=z$ and $x_2=1$ which makes the
set of polynomials $|a_0|+|a_1|x_1+|a_3|x_3+r_1(z_1,z_2,z_3)$, a
subset of the set of polynomials $|a_0|+(|a_1|+|a_3|)z+r(z)$. The
infimum taken over a subset is always bigger than or equal to the
infimum taken over the set itself. We obtained the inequality
(\ref{015}) by applying a similar argument because we can extract
the polynomial $|a_2|x_2^2$ from the set $r_1(x_1,x_2,x_3)$ and
$|a_2|x_2^2=|a_2|$ when $x_2=1$. The inequality (\ref{016}) was
obtained by choosing $r_1(x_1,x_2,x_3)$ in particular to be equal
to
\[
(a_0-|a_0|+a_2-|a_2|)x_2^2+(a_1-|a_1|)x_1x_2+(a_3-|a_3|)x_2x_3+q(x_1,x_2,x_3).
\]
\textbf{Case 2.} When $|a_0|> |a_1|$.\\

It is obvious from Case 1 that
\[
\left\| \begin{pmatrix} |a_0|&0\\
|a_3|&|a_0|+\dfrac{|a_1|}{4} \end{pmatrix} \right\|  \leq \left\| \begin{pmatrix} |a_0|&0 \\
|a_3|&|a_0|+\dfrac{|a_0|}{4} \end{pmatrix} \right\| \leq \left\| \begin{pmatrix} |a_0|&0\\
|a_0|+|a_3|&|a_0| \end{pmatrix} \right\| \,.
\]
Therefore,
\begin{align}
\left\| p(T_1,T_2,T_3) \right\| & \nonumber \leq \left\| \begin{pmatrix} |a_0|&0\\
|a_0|+|a_3|&|a_0| \end{pmatrix} \right\| \\& \nonumber = \inf \,
\||a_0|+(|a_0|+|a_3|)z+r(z)\|_{\infty, \overline{\mathbb D}} \\&
\nonumber \leq
 \inf \, \| |a_0|+|a_0|x_1+|a_3|x_3+r_1(x_1,x_2,x_3) \|_{\infty,
 \Lambda}\\& \nonumber \leq \inf \, \| |a_0|+|a_2|+|a_0|x_1+|a_3|x_3+r_1(x_1,x_2,x_3)\|_{\infty, \Lambda}
 \\& \nonumber
 = \inf \, \| |a_0|+|a_0|x_1+|a_2|x_2+|a_3|x_3+r_1(x_1,x_2,x_3) \|_{\infty, \Lambda}\\&
\label{eqn:ex6} \leq \| a_0+a_1x_1+a_2x_2+a_3x_3+q(x_1,x_2,x_3)
\|_{\infty,
 \Lambda} \\& \nonumber \leq \| a_0+a_1x_1+a_2x_2+a_3x_3+q(x_1,x_2,x_3)
\|_{\infty,
 \overline{\mathbb E}} \\& \nonumber =\|p(x_1,x_2,x_3) \|_{\infty, \overline{\mathbb E}}\,.
\end{align}
Here all notations used are as same as they were in Case 1 and we
obtained the inequality (\ref{eqn:ex6}) by choosing
$r_1(x_1,x_2,x_3)$ in particular to be equal to
\[
(a_0-|a_0|+a_2-|a_2|)x_2^2+(a_1-|a_0|)x_1x_2+(a_3-|a_3|)x_2x_3+q(x_1,x_2,x_3).
\]

\vspace{0.60cm}

\noindent \textbf{Acknowledgement.} The author is indebted to the
referee for his/her careful and rigorous reading of this article
and for pointing out an inaccuracy in the first version of the
paper. The author greatly appreciates the warm and generous
hospitality provided by Indian Statistical Institute, Delhi, India
during the course of the work.

\end{document}